\documentclass[12pt]{article} 

\usepackage{amsmath,amsfonts,bm,amssymb}
\usepackage{hyperref}
\usepackage{url}
\usepackage{algorithm}
\usepackage{algpseudocode}
\usepackage{mathtools}
\usepackage{bbm}
\usepackage{enumitem}
\usepackage{subcaption}
\usepackage{graphicx}
\usepackage{xcolor}
\numberwithin{equation}{section}

\title{Stochastic global optimization  of continuous functions
via random walks on
Grassmannians }

\author{ Kartik Gupta ({\tt kartik.kailash.gupta@gmail.com}), \\
         Stephen D. Miller\thanks{Supported by NSF grant CNS-2124692.}  \ ({\tt stephen.miller@yu.edu}), \\
      Pradeep Ravikumar ({\tt pkr@andrew.cmu.edu}), and \\
         Ramarathnam Venkatesan ({\tt venkie@microsoft.com})\\
      }

\newtheorem {theorem}{Theorem}

\newtheorem {corollary}[theorem]{Corollary}

\newtheorem {lemma}[theorem]{Lemma}

\newenvironment {proof}[1][Proof]{\noindent \textbf {#1.} }{\ \rule {0.5em}{0.5em}}

\begin{document}

\date{}

\maketitle

\begin{abstract}
We introduce a stochastic global optimization method based on random walks on Grassmannian manifolds.
To minimize a continuous objective $\ell:\mathbb{R}^d\rightarrow\mathbb{R}$, the method repeatedly samples random $k$-dimensional linear subspaces (with $k\ll d$), solves the resulting low-dimensional restrictions of these problems to these subspaces using an arbitrary black-box optimizer, and updates the iterate (which monotonically improves upon the previous iterate).
Unlike classical optimization analyses that rely on convexity, smoothness, Lipschitz bounds, or Polyak--{\L}ojasiewicz-type conditions, our convergence guarantees depend only on the geometric distribution of restricted minima across the $k$-dimensional subspaces passing through a given point in $\mathbb{R}^d$.
We identify a \textit{gap parameter}---an analogue of a spectral gap for random walks---that controls the rate at which the iterates approach the global minimum value.
Finally, we argue that the same analysis yields a \emph{blind-spot robustness} property:~sufficiently narrow, deep dips of the loss function (small-measure regions where $\ell$ spikes downward) have limited influence on the algorithm's trajectory, since they are unlikely to be encountered by random subspace sampling.
\end{abstract}

\section{Introduction}
\label{intro}

\subsection{Motivation:~global optimization under minimal regularity}

Modern machine learning and scientific computing routinely give rise to optimization problems that are high-dimensional, highly non-convex, and geometrically intricate.
In many of these settings, classical structural assumptions used to obtain convergence guarantees, such as convexity, smoothness, global Lipschitz bounds, or Polyak--{\L}ojasiewicz-type conditions, are either false, unknown, or hard to justify from first principles.
At the same time, much of the existing theory for popular methods such as (stochastic) gradient descent and its variants is inherently \emph{local}:~it characterizes convergence to stationary points under regularity assumptions on gradients and Hessians, and it is not designed to provide global progress guarantees for general objective functions.
We have in mind very complicated loss functions, e.g.,  coming from deep neural networks, which are at best continuous and lack the convexity properties usually required to guarantee progress over iterations.

 This motivates the question we study here:~\emph{can one design a practically implementable stochastic procedure with a provable tendency toward global minima, under assumptions as weak as continuity?}
Of course, global optimization is hard in the worst case, even for simple model classes, and the present discussion is only meant to give an intuitive understanding of our approach.
Rather than attempting to overcome worst-case hardness for \emph{all} loss functions, our perspective is to obtain a guarantee whose rate depends on a measurable geometric property (the ``gap parameter'' (\ref{eq:gapdef})) of the objective function, and to interpret this property as a form of tractability.  Our results are meaningless when the gap parameter is zero, and most interesting when the gap parameter is sufficiently large.  The latter is necessary in order to avoid numerous mathematical and complexity-theoretic difficulties that can arise, even in the context of smooth functions (where the set of critical points can be extremely complicated, e.g., a Cantor set \cite[Exercise~5.21]{rudin}).

\subsection{Random subspaces and Grassmannian randomness}

A familiar strategy for high-dimensional optimization is to reduce dimension, for example through coordinate descent or block coordinate methods.
Such approaches, however, are tied to a particular choice of coordinates.
Ideally, one might instead desire the coordinate-independent viewpoint of considering quantities that are well-defined on a manifold.  To make this approach feasible, we take the given coordinate system the optimization problem has been posed in, and create a problem on a manifold as follows.  Specifically, we consider \emph{all} rotations of the coordinate axes at once, and optimize on random \emph{linear subspaces} that cross the coordinate grid at generic angles (not just ones aligned with any particular chosen coordinate system of $\mathbb{R}^d$).
The collection of all $k$-dimensional subspaces of $\mathbb{R}^d$ forms the Grassmannian $\mathcal G_{k,d}$, a smooth manifold  equipped with a canonical rotation-invariant probability measure.
This geometry provides a natural specification for choosing subspaces at random, which is independent of the choice of coordinate axes.

\subsection{A stochastic global optimization procedure}

Our method proceeds by sequentially solving the restriction of the original problem to random $k$-dimensional subspaces (with $k\ll d$), using an arbitrary black-box optimization routine in these low-dimensional subspaces.
Each successive subspace is sampled to \emph{contain} the current iterate, so that the loss function values produced by the walk are non-increasing.
The resulting algorithm (Algorithm~\ref{alg:ptw}) in Section~\ref{results} can be viewed as a random walk on  $\mathcal G_{k,d}$.

A key novelty is that our algorithm's access to the continuous loss function is as a black box that  lacks access to derivatives.
Instead, it depends only on the \emph{distribution of restricted minima} across random subspaces.
This leads to a convergence estimate for continuous loss functions in terms of a single quantity --- the \emph{gap parameter} (Definition~(\ref{eq:gapdef})) --- which plays a role analogous to the spectral gap controlling the mixing rate of random walks on graphs.

\subsection{Blind-spot robustness to narrow pathological regions}

Beyond convergence, the same perspective yields a useful robustness notion.
If the loss function has an extremely narrow region where it dips sharply downward, a ``spike'' of very low values supported on a set of tiny measure, then a method based on random subspace sampling is unlikely to encounter this region.
Section~\ref{robust} formalizes this idea by analyzing a clipped objective function $\ell_{\alpha'}:=\max(\ell,\alpha')$, showing that the algorithm effectively optimizes the loss function \emph{outside} such small-measure blind spots.
This robustness is subtle:~it does not claim immunity to arbitrary perturbations, but it does explain why certain highly concentrated pathologies have limited impact on the algorithm's output.

\subsection{Organization of this paper}

Section~\ref{relwork} discusses connections to existing stochastic and derivative-free optimization techniques.
Section~\ref{prelim} contains the mathematical background on compact groups, Haar measure, and Grassmannians.
In Section~\ref{results} we present the random-subspace walk and prove our main global convergence guarantee.
Section~\ref{robust} develops the blind-spot robustness viewpoint.
We conclude in Section~\ref{disc} with a discussion of implications and potential applications.

We wish to thank Henry Cohn,  Michael Freedman,  Ravi Kannan, Bhargav Narayanan, and Nikhil Srivastava for their helpful comments.

\section{Related Work}
\label{relwork}

\textbf{Comparison to existing stochastic techniques.}
Given the current emphasis on applications of optimization in machine learning, we present a review of techniques focused around applications to computer science.  However, optimization as a whole is of course a far broader topic that contains other methods such as basin-hopping that are decisive in other fields (see \cite{Wales} for applications to chemistry and physics).
Most of the existing literature focuses on stochastic gradient descent or its popular variants such as Adam \cite{Kingma2014AdamAM} and AdaGrad \cite{JMLR:v12:duchi11a}.
In stochastic gradient descent one picks a random subset of the data, and computes the gradient of the loss function on this subset to update the parameters of the model.
Convergence for this scheme can be shown under assumptions like strong convexity \cite{NIPS2011_40008b9a}, the Polyak--{\L}ojasiewicz condition \cite{DBLP:conf/aistats/GowerSL21}, or convergence to stationary points for non-convex functions which satisfy an expected smoothness assumption \cite{khaled2022better}.
These convergence results rely crucially on structural assumptions on the loss function and, by design, they emphasize \emph{local} stationarity rather than global optimality.

Our approach is fundamentally different in that, rather than drawing randomness from subsampling the data, we select random linear subspaces of the Euclidean space in which the optimization problem is posed.
This differs from coordinate descent in that random linear subspaces are typically not aligned with the coordinate grid; they cross it at generic angles.
It also differs from random-direction search in that we randomize the \emph{subspace in which a subproblem is posed and solved}, while treating the solver within that subspace as a black box.
The only existing technique with superficial similarities is Dropout in deep learning \cite{JMLR:v15:srivastava14a}, but Dropout typically considers subsets (not subspaces) of parameters, and  is specialized to neural networks.
By contrast, the set of all $k$-dimensional subspaces is a smooth manifold (hence uncountable), and its canonical probability measure comes from the rotation group.

\textbf{Analysis via compact-group geometry.}
Analytically, our results rely on the fact that randomness is drawn from a homogeneous space (a quotient of a compact Lie group) equipped with Haar probability measure.
This enables convergence guarantees for continuous objective functions in terms of global distributional properties of restricted minima, rather than derivative-based regularity conditions such as smoothness or Lipschitzness.
\section{Mathematical ingredients}
\label{prelim}

\subsection*{Geometric setup and invariant sampling}

\textbf{Grassmannians and ``conditioned Grassmannians''.} Fix integers $d>k>1$.
Let $\mathcal G_{k,d}$ denote the Grassmannian of $k$-dimensional \emph{linear} subspaces (``$k$-planes'') of $\mathbb{R}^d$  (see \cite{Bendokat_2024}).
For any nonzero vector $x\in\mathbb{R}^d$, we let
\begin{equation}\label{eq:Gkmdmx}
\mathcal G_{k,d,x} \ :=\ \{\eta\in\mathcal G_{k,d} : x\in \eta\}
\end{equation}
be the submanifold of $k$-planes that contain $x$.  We call these ``conditioned Grassmannians''; they are related to generalized flag varieties and Schubert varieties.
These compact homogeneous manifolds, equipped with rotation-invariant measures, provide the state spaces for our random-subspace walk.

\textbf{Group viewpoint.} Let $G$ denote a Lie group and $H\subset G$ a closed subgroup, with quotient $G/H$.
Our primary example is $G=O(d)$ and $H=O(k)\times O(d-k)$, embedded in the standard way so that $H$ preserves the coordinate subspace
$V:=\mathrm{span}(e_1,\ldots,e_k)\subset\mathbb{R}^d$.
Then the quotient $G/H$ can be identified with $\mathcal G_{k,d}$ via $gH\mapsto gV$.
This identification lets us translate probabilistic statements on $\mathcal G_{k,d}$ into statements on the compact group $O(d)$.

\textbf{Invariant measures and sampling.} Any locally compact Lie group admits a left-invariant Haar measure $dg$, unique up to scaling \cite{nachbin1976haar}.
After normalizing so that $G$ has measure 1, $dg$ becomes a probability measure when $G$ is compact.
Pushing $dg$ forward under $g\mapsto gV$ yields the canonical $O(d)$-invariant probability measure on $\mathcal G_{k,d}$.
Likewise, the ``conditioned Grassmannians'' $\mathcal G_{k,d,x}$ (\ref{eq:Gkmdmx}) inherit unique probability measures invariant under the stabilizer subgroup of $x$
(see e.g.~\cite[Chapter 1]{sepanski2006compact} and \cite{gallier2020differential} for quotient-measure constructions).
Throughout, the measure of a measurable subset $A$ under its ambient space's Haar measure will be denoted by $|A|$.

\textbf{The objective function $\phi$ and its restricted minima.} Let $\ell:\mathbb{R}^d\rightarrow\mathbb{R}$ be the continuous objective (loss) function we wish to minimize.  We tacitly assume $\ell$ is coercive, so that it attains minima on each $\eta\in \mathcal G_{k,d}$.
Our method queries $\ell$ only through low-dimensional restrictions:~given a $k$-plane $\eta$, we assume access to a black box that returns
(exactly or approximately) a minimizer of $\ell$ restricted to $\eta$.
To study the distribution of these restricted optima, we define
\begin{equation}
\label{eq:phidef}
\phi(\eta) \ = \ \min_{y\in \eta}\ell(y).
\end{equation}
Note that
\[
\min_{y\in\mathbb{R}^d}\ell(y)\ =\ \min_{\eta\in\mathcal G_{k,d}}\phi(\eta),
\]
i.e., the global minimum value of $\ell$ coincides with the global minimum of $\phi$.

\subsection{Probabilities for random values being near minima of functions on compact groups.}

Algorithm~\ref{alg:ptw} repeatedly samples random subspaces (elements of $\mathcal G_{k,d}$ or $\mathcal G_{k,d,x}$) and keeps the best among $T$ samples.
To analyze such steps, we need a quantitative statement of the following form:~if a bounded function on a compact homogeneous space is not almost constant under the invariant measure, then random sampling has a non-negligible chance of producing a value that is \emph{meaningfully below its maximum}.
When applied to $\phi(\eta)=\min_{y\in\eta}\ell(y)$, this becomes a lower bound on the probability of drawing a subspace whose restricted optimum is closer to the global optimum than it is to the maximal (worst) restricted optimum.

The lemma below is a simple variance-to-level-set bound on a compact Lie group equipped with Haar probability measure.
In its most transparent specialization, suppose $f:G\to\mathbb{R}$ has been rescaled so that its range is contained in $[0,1]$.
If the $L^2$-deviation $\|f-\bar f\|_2$ is bounded away from zero, then $f$ cannot concentrate near $1$ everywhere, and a random draw $g\sim dg$ has a positive chance of landing in a set where $f(g)\le 1-\delta$ for an explicit $\delta$.
Lemma~\ref{lem:measurebounds} records a slightly more flexible version that also allows an exceptional set $U$ on which $f$ dips below a threshold $\alpha'$; this flexibility is what we later exploit in our blind-spot robustness analysis (Section~\ref{robust}).

Concretely, in the special case $\alpha'=\alpha_\text{min}$ (so $U=\varnothing$) and after rescaling $f$ so that $\alpha_\text{min}=0$ and $\alpha_\text{max}=1$, the parameter in Lemma~\ref{lem:measurebounds} simplifies to $\delta=\tfrac{1}{\sqrt{2}}\|f-\bar f\|_2$, and (\ref{eq:measureboundspunchline}) becomes
\[
\mathbb{P}_{g\sim dg}\!\left[f(g)\le 1-\delta\right]\ \ge\ \delta^2.
\]
This provides an explicit way to read the lemma as a statement about the likelihood of sampling a value that is not ``near the top'' of the range.

Our presentation is motivated by (and can be derived from) Kazhdan's Property (T) \cite[\S3.1]{rogawski1994discrete}, which holds trivially for compact groups such as $G=O(d)$.

\begin{lemma}
\label{lem:measurebounds}  Let $G$ be a compact Lie group with Haar measure $dg$, normalized so that $|G|=1$.  Let $f:G\rightarrow \mathbb{R}$ be continuous with mean $\bar{f}=\int _G f(g)dg$.   Let $\alpha_\text{min}$ and $\alpha_\text{max}$ denote, respectively, the minimum and maximum values of $f$, let $\alpha' \in (\alpha_\text{min},\alpha_\text{max})$, and let
\begin{equation}\label{eq:Udef}
U = \{ g : f(g) < \alpha'\}.
\end{equation}
Assume that $\alpha'$ is chosen so that  $|U|< \frac{\|f-\bar{f}\|^2_2}{(\alpha_\text{max}-\alpha_\text{min})^2}$,
which ensures that
\begin{equation}\label{eq:deltaafterUdef}
\delta=\sqrt{\frac 12 \frac{\|f-\bar{f}\|^2_2}{(\alpha_\text{max}-\alpha')^2}-
\frac{|U|}{2}\frac{(\alpha_\text{max}-\alpha_\text{min})^2}{(\alpha_\text{max}-\alpha')^2}}
\end{equation}
exists and is positive.
Then
\begin{equation}\label{eq:measureboundspunchline}
\Big|\big\{ g  :  f(g)-\alpha'\le (1-\delta)(\alpha_\text{max}-\alpha')\big\}\Big| \ \ge \ \delta^2.
\end{equation}
That is, random values of $f$, with probability at least $\delta^2$, are not in the uppermost $\delta$-proportion of the range $[\alpha',\alpha_\text{max}]$.
\end{lemma}

\paragraph{How to read Lemma~\ref{lem:measurebounds}.}
Think of $G$ as a probability space and $f(g)$ as a random variable when $g$ is drawn from Haar measure.
The parameters $\alpha_\text{min}$ and $\alpha_\text{max}$ describe the range of $f$.
The set $U=\{g:f(g)<\alpha'\}$ isolates an exceptional region where $f$ takes unusually small values.
The conclusion (\ref{eq:measureboundspunchline}) lower bounds the probability mass of points where $f$ is \emph{not} in the very top of its range above $\alpha'$:~with probability at least $\delta^2$, a random sample satisfies
\[
f(g)\ \le\ \alpha_\text{max}-\delta(\alpha_\text{max}-\alpha').
\]
To connect with the comment above, in the special case $\alpha'=\alpha_\text{min}$ (so $U=\varnothing$), this says that random sampling reaches values a positive fraction of the way from the worst value $\alpha_\text{max}$ toward the best value $\alpha_\text{min}$.

\begin{proof}
Augmenting (\ref{eq:Udef}), let
\begin{equation}\label{eq:msbds1}
\aligned
    S  \ & = \  \{g : f(g) \ge \alpha_\text{max}-\delta(\alpha_\text{max}-\alpha')\} \ \ \text{and}\\
    T \ & = \ \{g : \alpha' \le f(g) <\alpha_\text{max}-\delta(\alpha_\text{max}-\alpha')\},
\endaligned
\end{equation}
so that $|S|+|T|+|U|=1$.
The key point is that $f-\bar{f}$, having mean zero, is orthogonal to all constants, including $\alpha_\text{max}-\bar{f}$.  Thus
\begin{equation}\label{eq:msbds2}
\aligned
\|f-\bar{f}\|_2^2 \ 
& \le \  \|\alpha_\text{max}-\bar{f}\|_2^2 \, + \, \|f-\bar{f}\|_2^2 \ = \
\|\alpha_\text{max}-f\|_2^2 \\
& \le \ \delta^2(\alpha_\text{max}-\alpha')^2 |S| +
(\alpha_\text{max}-\alpha')^2|T| + (\alpha_\text{max}-\alpha_\text{min})^2|U|
\\
& \le \ (\alpha_\text{max}-\alpha')^2(|T|+\delta^2) \, + \, |U|(\alpha_\text{max}-\alpha_\text{min})^2\,,
\endaligned
\end{equation}
where in the last step we have used $|S|\le 1$.
 Were $|T|<\delta^2$, then the right-hand side of (\ref{eq:msbds2}) would be less than the left-hand side; hence $|T|\ge \delta^2$, i.e., (\ref{eq:measureboundspunchline}).
\end{proof}

Lemma~\ref{lem:measurebounds} is stated in a fairly general form (in particular, allowing an exceptional set $U$).
This generality becomes useful later in the blind-spot robustness analysis (Section~\ref{robust}), where we want to ``clip away'' tiny regions on which the loss function dips far below its typical values.
For the main global-convergence analysis, however, we will invoke Lemma~\ref{lem:measurebounds} primarily through two simplifications:

\begin{enumerate}
\item \textbf{No exceptional set:} $\alpha'=\alpha_\text{min}$.
As mentioned above, this forces $U=\varnothing$ and (\ref{eq:deltaafterUdef}) simplifies to
\[
\delta\ =\ \frac{1}{\sqrt{2}}\frac{\|f-\bar{f}\|_2}{\alpha_\text{max}-\alpha_\text{min}}
\]
(the constant $\tfrac{1}{\sqrt{2}}$ can be improved).

\item \textbf{Passing to homogeneous spaces.}
In our applications, the relevant functions live on homogeneous spaces $G/H$ rather than on $G$ itself.
If $f$ is right-invariant under a closed subgroup $H$ of $G$, then it descends to a function on the quotient manifold $G/H$, and the same measure-theoretic conclusions hold with respect to the unique $G$-invariant probability measure on $G/H$
(this is the push-forward measure mentioned above; see also \cite[Theorem 1.9 and Remark p.93]{helgason2022groups}).
\end{enumerate}

In the setting of this paper we take $G=O(d)$ and $H=O(k)\times O(d-k)$, so $G/H\simeq \mathcal G_{k,d}$.
We will also need the ``conditioned Grassmannians'' $\mathcal G_{k,d,x}$ of $k$-planes constrained to contain a given nonzero vector $x$ defined in (\ref{eq:Gkmdmx}). They play an important role in our analysis, because each iteration of Algorithm~\ref{alg:ptw} samples $\eta$ from $\mathcal G_{k,d,x_i}$, the set of $k$-planes containing the current iterate $x_i$.
For fixed $x\neq 0$, $\mathcal G_{k,d,x}$ can be identified (non-canonically) with the lower-dimensional quotient manifold $O(d-1)/(O(k-1)\times O(d-k))$, reflecting the freedom to choose a $(k-1)$-plane inside the orthogonal complement of $x$.
In particular, both $\mathcal{G}_{k,d}$ and $\mathcal{G}_{k,d,x}$ are homogeneous spaces carrying unique invariant probability measures.

Applying Lemma~\ref{lem:measurebounds} to the restriction of $\phi$ to either $S=\mathcal G_{k,d}$ or $S=\mathcal G_{k,d,x}$ yields a constant $\delta$ (defined below) such that a single uniform sample has probability at least $\delta^2$ of producing a subspace $\eta$ with $\phi(\eta)$ bounded away from the largest value on $S$.
Taking the best among $T$ independent samples amplifies this success probability to $1-(1-\delta^2)^T$.
We record this amplified sampling guarantee as the following corollary.

\begin{corollary}\label{cor:submfld}

Let $S$ be equal to either $\mathcal G_{k,d}$ or one of  its submanifolds $\mathcal{G}_{k,d,x}$, for some nonzero $x\in \mathbb{R}^d$.  Let $M_S$ (resp., $m_S$), denote the maximum (resp., minimum) value of $\phi$ restricted to $S$, and let
\begin{equation}\label{eq:deltaS}
\delta \ \ = \ \frac{1}{\sqrt{2}}\frac{\big\|\,\phi|_S-
\overline{\phi|_S}\,\big\|_{S,2}}{M_S-m_S},
\end{equation}
where $\phi|_S$ denotes the restriction of $\phi$ to $S$, $\overline{\phi|_S}$ its average over $S$, and the norm is the $L^2$-norm on $S$.  Then given $T$ points $\eta\in S$ sampled uniformly at random with respect to $S$'s probability measure, the probability that least one will satisfy
\begin{equation}
\label{eq:corpunch}
\phi(\eta) - m_S \ \le  (1-\delta)\left(M_S-m_S\right)
\end{equation}
is at least $1-(1-\delta^2)^T$.
In particular, the value of $\phi$ on this $\eta$ will be bounded away from the maximum value $M_S$.
\end{corollary}

The fundamental nature of this Lemma should be compared with results like the Markov inequality or the Chebyshev inequality, which instead give a non-trivial lower bound on the probability of getting a value close to the mean via random sampling.

Though it is overkill  to appeal to it for the analysis here, these sampling results were heavily motivated by Kazhdan's Property (T), a very important and extensively studied notion used in analyzing random walks on certain noncompact Lie groups \cite{bekka2008kazhdan} that holds trivially for all compact Lie groups.

\textbf{Generality of Lemma~\ref{lem:measurebounds}.}  Though our primary application is to the orthogonal group $G=O(d)$, there are other  compact groups (which also have Haar probability measures) of interest in the theory of optimization.  For example, one may take any finite group or the $n$-dimensional torus $(\mathbb{Z}\backslash\mathbb{R})^n$, which is --- issues of boundary identifications aside --- nearly the $n$-dimensional hypercube.

\section{Stochastic global optimization:~algorithm, guarantees, and examples}
\label{results}

In this section, we present our random-subspace walk, a stochastic global optimization procedure applicable to any continuous loss function, and we prove a global convergence guarantee for this general setting.

\subsection{Random-subspace walk}

Our objective is \emph{global} minimization:~we seek a point $x^*\in\mathbb{R}^d$ satisfying
\[
x^*\in\arg\min_{x\in\mathbb{R}^d}\ell(x).
\]
Many popular methods in modern non-convex optimization are analyzed through local stationarity conditions (e.g., small gradients) and therefore do not directly track progress toward the global minimum value.
In contrast, our random-subspace walk is built so that the analysis follows the gap $\ell(x_i)-\alpha_\text{min}$ \emph{itself}, yielding a global convergence guarantee under only continuity (cf.~Theorem~\ref{thm:conv}).
Of course, when run for limited time the method may return a point near a non-global local minimum; our robustness discussion later in the paper shed light on which local minima can realistically attract the walk.

Assume that we are given access to a black box  for solving (or at least approximating) the same optimization problem  restricted to a smaller dimensional $k$-plane $\eta\subset\mathbb{R}^d$, i.e., we can find an $x^*_\eta$ such that
\[x_\eta^*\in\arg\min\limits_{x\in\eta}\ell(x),\]
or at least a good approximation to one.
Our random walk attempts to use this   black box repeatedly, over some sequence of $k$-planes $\eta_1,\eta_2,\ldots$, to find a candidate for $x^*$. Among various other possibilities, a natural random walk is given as follows:~starting with some $x_1\in\mathbb{R}^d$, sample a random $k$-plane $\eta_1\in \mathcal{G}_{k,d,x_1}$ containing $x_1$;
at the $i$-th step, find some  $x_{i+1}\in \arg\min_{x\in \eta_i}\ell(x)$ (or an approximation to such a minimizer) and then
sample a random $k$-plane $\eta_{i+1}\in \mathcal{G}_{k,d,x_{i+1}}$;
 stop the algorithm after $N$ iterations. We state one such variant of this more formally as Algorithm \ref{alg:ptw}, with a concrete suggestion for the stopping rule's number of steps $N$.

\begin{algorithm}
\caption{Random-subspace walk (stochastic global optimization)}\label{alg:ptw}
\begin{algorithmic}[1]
\Require A continuous loss function $\ell:\mathbb{R}^d\rightarrow\mathbb{R}$, integers $d>k>1, T\ge 1$, accuracy parameter $\epsilon_a>0$, and a black-box $\mathcal B$ which takes $k$-planes $\eta$ as input and returns a minimizer of $\ell$ restricted to $\eta$ as output.
\State Set $x_0$ to be some fixed element of $\mathbb{R}^d$, e.g., 0
\State Sample $\eta^{(1)},\ldots,\eta^{(T)}$ independently and uniformly at random from $\mathcal G_{k,d}$
\State Choose any  $x_1 \in \arg\min_{y\in \{{\mathcal B}(\eta^{(1)}),\ldots,{\mathcal B}(\eta^{(T)})\}}\ell(y)$
\State Let $i=1$
\While{$\ell(x_{i-1})-\ell(x_{i})> \epsilon_a$}
    \State Sample $\eta^{(1)},\ldots,\eta^{(T)}$ uniformly from $\mathcal G_{k,d,x_{i}}$ (cf.~(\ref{eq:Gkmdmx}))
    \State Let $x_{i+1}=\arg\min_{y\in \{{\mathcal B}(\eta^{(1)}),\ldots,{\mathcal B}(\eta^{(T)})\}}\ell(y)$
    \State $i=i+1$
\EndWhile
\State \Return $x_{i}$
\end{algorithmic}
\end{algorithm}

This is a  natural random walk from the computational complexity theory perspective.   It leverages the ability to solve several lower-dimensional random problems when attempting to solve a higher-dimensional problem.

\subsection{Analysis of convergence properties of Algorithm \ref{alg:ptw}}
\label{subsec:conv-analysis}

We now study the convergence properties of Algorithm \ref{alg:ptw}.
Let $\phi$ be as defined in (\ref{eq:phidef}) and let  $\alpha_{\text{max}}$ (resp., $\alpha_\text{min}$) denote its global maximum (resp., minimum) value.
According to Corollary~\ref{cor:submfld} applied to $S=\mathcal{G}_{k,d}$, Step 3 succeeds
in finding some $\eta_1\in \mathcal{G}_{k,d}$ and an element $x_1\in \eta_1$ satisfying
\begin{equation}
    \label{eq:analysisinitial}
    \ell(x_1) - \alpha_{\text{min}} \ \le \ (1-\delta) \left(\alpha_\text{max} - \alpha_\text{min}  \right),
\end{equation}
with probability at least $\delta_T=1-(1-\delta^2)^T$, where
\begin{equation}\label{eq:deltadef}
\delta \ = \ \frac{1}{\sqrt{2}}\frac{\big\|\phi-\bar{\phi}\big\|_2}{\alpha_\text{max}
-\alpha_\text{min}}
\end{equation}
 and $\bar{\phi}$ is the mean of $\phi$.
At each iterative step, we apply Corollary~\ref{cor:submfld} to $S=\mathcal G_{k,d,x_i}$, with $\delta(x_i)$ as given in (\ref{eq:deltaS}), and $\delta(x_i)_T$ likewise denoting $1-(1-\delta(x_i)^2)^T$.  In this context, $\phi$'s minimum value $m_S$ on $S$ must equal the global minimum $\alpha_\text{min}$, because any vector in $\mathbb{R}^d$ (a global minimizer included) must lie inside some element of $S$. At the same time,
$\phi$'s maximum $M_S$ on $S$ cannot exceed  $\ell(x_i)$, because each $\eta\in S$ already contains $x_i$.
Thus with probability at least $\delta(x_i)_T$, the algorithm finds some $\eta_{i+1}\in \mathcal{G}_{k,d,x_i}$ and some $x_{i+1}\in \eta_{i+1}$ with
\begin{equation}
    \label{eq:analysisiter}
    \ell(x_{i+1}) - \alpha_{\text{min}} \ \le \ (1-\delta(x_i))\left(\ell(x_i)-\alpha_\text{min}  \right).
\end{equation}
We deduce the inequality
\begin{equation}
\label{eq:analysiscascade}
 \ell(x_{n+1}) - \alpha_{\text{min}} \ \le \ (1-\delta)\left( \prod_{i=1}^n(1-\delta(x_i))\right) (\alpha_\text{max}-\alpha_\text{min})
\end{equation}
by telescoping the inequalities  (\ref{eq:analysisinitial})-(\ref{eq:analysisiter}).

To simplify (\ref{eq:analysiscascade}), we define ``gap parameter'' of the loss function $\ell$ as
\begin{equation}
    \label{eq:gapdef}
    \Theta_\ell \ := \  \inf_x \frac{\bigg\|\phi|_{\mathcal{G}_{k,d,x}}-
    \overline{\phi|_{\mathcal{G}_{k,d,x}}}\bigg\|_{2,\mathcal{G}_{k,d,x}}}{\max(\phi|_{\mathcal{G}_{k,d,x}}) - \min(\phi|_{\mathcal{G}_{k,d,x}})}\,,
\end{equation}
where the infimum  is taken over all $x$ for which the restriction of $\phi(\eta):=\min_{x\in\eta}\ell(x)$ to $\eta\in\mathcal{G}_{k,d,x}$ is nonconstant, i.e., $x$ which are not already global minimizers of $\ell$.  Proving good lower bounds for $\Theta_\ell$ (or even that it is positive) is beyond the scope of the present paper, and likely depends on  specific information about the loss function $\ell$.
Therefore each $\delta(x_i)\ge \frac{1}{\sqrt{2}}\Theta_\ell$ and  the probability of success at each stage is at least $\delta(x_i)_T\ge \Theta_{\ell,T}$, where
\begin{equation}\label{eq:deltabiggerTheta}
\Theta_{\ell,T} \ = \  1-(1-\Theta_\ell^2/2)^T.
\end{equation}
We conclude:
\begin{theorem}
\label{thm:conv}
Let $\ell:\mathbb{R}^d\rightarrow\mathbb{R}$ be a continuous loss function with $\Theta_\ell>0$.  Then Algorithm~\ref{alg:ptw} returns, with probability  at least
$$\left(1-(1-\delta^2)^T\right) \left(1-(1-\Theta_\ell^2/2)^T \right)^n,$$
a point $x_{n+1}\in \mathbb{R}^d$ satisfying
\begin{equation}\label{eq:thmalganalysispunch}
\ell(x_{n+1})-\alpha_\text{min} \ \le \ (1-\delta)(1-2^{-1/2}\Theta_\ell)^n (\alpha_\text{max}-\alpha_\text{min}).
\end{equation}
In particular, the successive differences of loss values satisfy the bounds
\begin{equation}\label{eq:successiveloss}
\aligned
  \ell(x_{n})-\ell(x_{n+1}) \ & = \  (\ell(x_n)-\alpha_\text{min})-(\ell(x_{n+1})-\alpha_\text{min}) \\
  & \le \  \ \ell(x_n)-\alpha_\text{min} \\
  & \le \  (1-\delta)  (1-2^{-1/2}\Theta_\ell)^{n-1}(\alpha_\text{max}-\alpha_\text{min}),
\endaligned
\end{equation}
so that Algorithm~\ref{alg:ptw} terminates within $n=\Big\lceil \frac{\log\big(\frac{\epsilon_a}{(1-\delta) (\alpha_\text{max}-\alpha_\text{min}) }\big)}{\log(1-2^{-1/2}\Theta_\ell)}\Big\rceil+1$ iterations.  Moreover, if
\begin{equation}\label{eq:Tbd}
T \ = \ \Big\lceil \max\left(\frac{\log(1-\frac{1}{\sqrt{2}})}{\log(1-\delta^2)},
\frac{\log(1-2^{-1/(2n)})}{\log(1-\Theta_\ell^2/2)} \right)\Big\rceil,
\end{equation}
the probability of success is at least 1/2.

\end{theorem}

{\bf Remarks:}


\begin{enumerate}
    \item Unlike analysis of algorithms in the non-convex optimization literature that track consecutive iterates $\ell(x_{n})-\ell(x_{n+1})$, and use local regularity properties of $\ell$ to show they are small, our analysis is instead  {\it non-local} in the sense that at each step we directly track progress towards the global minimum $\alpha_\text{min}$, thus only indirectly  obtaining the bound on iterates in (\ref{eq:successiveloss}) as a byproduct of (\ref{eq:thmalganalysispunch}).  Furthermore, the search space for consecutive iterates $\{x_i\}$ is in principal global, rather than simply nearby previous iterates.
    \item The only regularity assumption made on the loss function $\ell$ is continuity.  In particular, we do not require any bounds on the Lipschitz constant of $\ell$, as is the case in the analysis of convergence of algorithms such as gradient descent.  This is because the relevant quantities such as the ``gap parameter'' $\Theta_\ell$ from (\ref{eq:gapdef}) are defined through integrals rather than derivatives.
    \item The gap parameter  $\Theta_\ell$ is defined in terms of the ratio of two ways to measure the variation of a function:~the $L^2$-distance  from its mean, and the largest spread among its values.  Both quantities are nonzero for nonconstant continuous functions, and each  remains unchanged if the mean is subtracted from the function.  We were very motivated by Kazhdan's Property (T) because of its formulations in terms of harmonic analysis on spaces of mean-zero functions, and our results can be derived from Property (T) since the latter automatically holds  for compact groups.
    \item There are parallels between the  gap parameter $\Theta_\ell$ and the spectral gap of the laplacian on a graph.  One of the very important and early applications of Kazhdan's Property (T) was the first explicit construction of  expander graphs in \cite{margulis1973explicit}. By virtue of their spectral gap (the difference between the largest two eigenvalues of the graph adjacency matrix), expander graphs have very good mixing properties, i.e., a random walk on one quickly becomes evenly distributed  across the graph \cite{rogawski1994discrete}.  This can be understood in terms of how the adjacency matrix of an expander graph shrinks functions orthogonal to constants, which connects to the previous remark.  The gap parameter dictates convergence for a similar reason (compare (\ref{eq:successiveloss}) with \cite[Lemma~2.1]{DBLP:conf/ants/MillerV06} and its proof).
        This common source of operating principles is the reason why we call $\Theta_\ell$ the gap parameter of $\ell$.
\end{enumerate}

\subsection{Experimental results}\label{sec:experiments}

In order to test the ability of Algorithm~\ref{alg:ptw} to solve optimization problems, we performed some experiments on J.J.~Thomson's classic problem of minimizing the Coulomb potential over all configurations of $N$ points on the unit sphere in $\mathbb{R}^{n+1}$ \cite{Cohn1}.  We selected this problem because it has been heavily studied over the last century, and is notorious for having a complicated landscape of myriad local minima \cite{Wales}.

We applied stereographic projection to convert to a minimization problem over $\mathbb{R}^{Nn}$.  Although we make no claims that Algorithm~\ref{alg:ptw} has superior speed to known approaches, we did find that  we could achieve known records on this problem \cite{Cohn2} provided adequate CPU time --- even with various choices of the ``black-box'' solver in $k$-dimensional subspaces.

The details of the implementation reported on here are as follows.  We start from a random point in $\mathbb{R}^n$ selected by independently sampling each coordinate from the normal distribution with mean zero and variance one.  At step $i$, we choose $k-1$ vectors at random from the same Gaussian distribution on $\mathbb{R}^n$, that together with $x_i$ span a $k$-plane $\eta_i$.  Next, we perform the following variant of random descent within $\eta_i$, starting from $x_i$:~for $j=1,\ldots,m$ we compute a random vector  of length $2^{-j}$ in $\mathbb{R}^k$, map it to a vector $v\in\eta_i$ using the chosen spanning basis, and move by $v$ if it lowers the value of the loss function.  This results in a candidate minimizer $x_{i+1}$, and the algorithm then repeats at step $i+1$.

Algorithm~\ref{alg:ptw} performed well in many situations, the following two being the most notable.

\bigskip
\noindent
{\bf 42 points in $\mathbb{R}^7$:} Using the values $k=150$, $m=20$, and $T=1$, we found a configuration with Coulomb energy $\approx227.410$, which ties the existing record in \cite{Cohn2} that was announced in 2021, and is smaller than the previous record of $\approx227.436$ set in \cite{Cohn1}.  This demonstrates that Algorithm~\ref{alg:ptw} can be competitive in quality with existing approaches.

\bigskip
\noindent {\bf Robustness in $\mathbb{R}^4$:} The optimal Coulomb energy for $N=120$ points in $\mathbb{R}^4$ is known be exactly 5395.  In \cite[Table 11]{Cohn1} the authors perform a fascinating study of how likely it is that their optimization method achieves various local minima.  In particular, they found that their algorithm got stuck at different local, non-global minima in 6.8\% of their 200,000 trials.

We performed the same experiment, but using Algorithm~\ref{alg:ptw} with the parameter choices $k=4$, $m=50$, and $T=1$, and letting the program run for exactly 24 hours of wall time on the Rutgers Amarel High Performance Cluster (approximately $1.5\times 10^7$ iterative steps).  In our 1,000 trials, only 1.9\% got stuck at local minima other than the global minimum.  Thus Algorithm~\ref{alg:ptw} cut the rate of failure to achieve the global minimum by over a factor of three.

\section{Blind-spot robustness}
\label{robust}

Theorem~\ref{thm:conv} shows that the iterations of Algorithm~\ref{alg:ptw} converge towards the global minimum $\alpha_\text{min}$ of the loss function $\ell(\cdot)$, at a rate controlled by the gap parameter $\Theta_\ell$ (cf.~(\ref{eq:thmalganalysispunch})).  However, in practice (given time constraints) the algorithm might yield an output near a different local minimum.  One particular such scenario is when $\ell(\cdot)$ spikes dramatically downward on a set of small measure, which would be a blind spot for random sampling.  In this section we analyze this scenario, in which one is only concerned with minimizing the loss function away from such minuscule, invisible subsets.

Consider a value $\alpha'$ between  the extreme values $\alpha_\text{min}$ and $\alpha_\text{max}$ of the function $\phi$ defined in (\ref{eq:phidef}), all three of which may be unknown.  We define an auxiliary loss function $\ell_{\alpha'}:=\max(\ell, \alpha')$, which is itself continuous (note that even if $\ell$ is smooth, $\ell_{\alpha'}$ may fail to inherit differentiability at the level set $\ell=\alpha'$, but this is not relevant to the present analysis).
By applying Theorem~\ref{thm:conv} to the loss function $\ell_{\alpha'}$ instead of $\ell$, we deduce the following estimates:

\begin{theorem}
\label{thm:convwithblindspot}
Let $\ell:\mathbb{R}^d\rightarrow\mathbb{R}$ be a continuous loss function,  let $\Theta_{\ell_{\alpha'}}$ denote the gap parameter of the loss function $\ell_{\alpha'}:=\max(\ell, \alpha')$ for some $\min(\ell)\le \alpha'\le \alpha_\text{max}$, and let
$\delta_{\alpha'}:=\frac{1}{\sqrt{2}}\frac{\| \phi_{\alpha'}-\overline{\phi_{\alpha'}} \|_2}{\alpha_\text{max}-\alpha'}$, where $\phi_{\alpha'}(\eta):=\min_{x\in\eta}\ell_{\alpha'}(x)=\min(\phi(\eta),\alpha')$ in analogy with (\ref{eq:phidef}), and $\overline{\phi_{\alpha'}}$ is its mean over ${\mathcal G}_{k,d}$.  Then the $n$-th round of the iteration described in Algorithm~\ref{alg:ptw} returns, with probability at least
$$\left(1-(1-(\delta_{\alpha'})^2)^T\right) \left(1-(1-(\Theta_{\ell_{\alpha'}})^2/2)^T \right)^n,$$
a point $x_{n+1}\in \mathbb{R}^d$ satisfying
\begin{equation}\label{eq:thmalganalysispunchB}
\ell(x_{n+1})-\alpha' \ \le \ (1-\delta)(1-2^{-1/2}\Theta_{\ell_{\alpha'}})^n (\alpha_\text{max}-\alpha'
),
\end{equation}
for all $n\in\mathbb{N}$.
\end{theorem}

Theorem~\ref{thm:convwithblindspot} specializes to Theorem~\ref{thm:conv} when $\alpha'= \alpha_{\text{min}}:=\min(\ell)$, and is meaningless for $\alpha\ge \alpha_{\text{max}}$.  It is therefore important to understand how the quantities $\Theta_{\ell_{\alpha'}}$ and $\delta_{\alpha'}$, which are defined   in terms of ratios of $L^2$-norms to $L^\infty$-norms (e.g., in (\ref{eq:gapdef})), vary as $\alpha'$ increases from $\alpha_\text{min}$ to $\alpha_\text{max}$.  The denominator is of course  nonincreasing in $\alpha'$.  At the same time, the numerators -- being  {\it averages} over many values -- might not change much if the set
$$\{\eta \in \mathcal G_{k,d}  \mid  \phi(\eta)<\alpha' \}$$
is tiny.  We conclude that when the loss function has sufficiently steep dips below $\alpha'$, we can expect that the constants in Theorems~\ref{thm:conv} and Theorem~\ref{thm:convwithblindspot} to be fairly similar.  This can be quantified and extended using the full formulation  of Lemma~\ref{lem:measurebounds}, which relates the effect of the parameter $\alpha'$ and the size of the set (\ref{eq:Udef}) to $\delta$ in a way very similar to its use in analyzing Algorithm~\ref{alg:ptw}.

When the clipped loss function $\ell_{\alpha'}$ differs from $\ell$ on only a small set, a random walk may have difficulty distinguishing them.  Blindspot robustness is, in spirit, closely related to the set-hitting and set-avoidance problems that arise in the study of random walks on expander graphs, with applications in cryptography and random number generation. Let $A$ be a measurable subset whose measure is proportion $\mu$ of the full space's measure. Given a random walk of length $L$, the empirical occupation measure of $A$ is defined as $\frac{m}{L}$, where $m$ denotes the number of times the walk visits the set $A$. For symmetric random walks on expander graphs, it is known (see, e.g., \cite{Gillman}) that the occupation measure concentrates sharply around $\mu$:~specifically, Chernoff-type bounds hold for the deviation of $\frac{m}{L}$ from $\mu$.
A symmetric walk can loop back to earlier steps, and  converge to the uniform distribution.  However, in many naturally arising applications the walk needs be asymmetric (e.g., the Pollard Rho walk for discrete logarithms \cite{DBLP:conf/ants/MillerV06}) and has no short cycles; adapting  Chernoff bounds is  non-trivial. PTW walks do not achieve uniform distribution since they converge to optima, and almost never loop, so a  proof technique must be fundamentally different.

\subsection{Potential application:~adversarial data poisoning}

One setting where blind-spot robustness may be relevant is \emph{adversarial data poisoning}:~an attacker modifies or injects a small fraction of the training data with the goal of steering learning toward a model with degraded (or targeted) test-time behavior.
A large literature studies poisoning attacks and defenses in modern machine learning; see \cite{10.1145/3551636,10.1145/3585385} for surveys.

The clipped-loss viewpoint above suggests a complementary, geometry-driven mechanism.
Many poisoning and backdoor constructions can be interpreted as creating a region of parameter space where the \emph{empirical} loss function dips anomalously low because it overfits a tiny corrupted subset, while remaining comparatively higher elsewhere.
If this ``too-good-to-be-true'' basin is highly concentrated (small measure under the random-subspace exploration distribution induced by $\mathcal G_{k,d}$ and $\mathcal G_{k,d,x}$), then the walk is unlikely to encounter it, and Theorem~\ref{thm:convwithblindspot} implies that optimizing the clipped objective function $\ell_{\alpha'}$ behaves similarly to optimizing the original loss function outside that basin.

This should be interpreted as a mechanistic robustness guarantee rather than a worst-case security statement.
Our results do \emph{not} provide general immunity to poisoning:~an adversary could in principle create a broad region of deceptively low loss, or otherwise alter the loss landscape on a non-negligible fraction of random subspaces, in which case the blind-spot assumption is violated.
Understanding when a concrete poisoning model leads to a genuinely small-measure downward spike for the induced loss landscape is an interesting direction for future work, and likely requires combining our geometric viewpoint with robust-learning and robust-statistics techniques \cite{prasad2020robust,10.1145/3055399.3055491}.

Finally, note that this effect may become more pronounced for smaller values of $k$:~a perturbation that only ``rewards'' a very low-dimensional cone of directions is harder to encounter when the walk explores through random, low-dimensional $k$-planes.

\section{Discussion}
\label{disc}

In this paper, we have presented a stochastic \emph{global} optimization algorithm built from random walks on Grassmannians.
The method repeatedly solves low-dimensional random restrictions of a given objective function and monotonically improves the loss function value.
Our main convergence theorem provides a guarantee whose only  regularity requirement on the objective function is continuity.
The analysis highlights a single quantity --- the \emph{gap parameter} $\Theta_\ell$ --- as controlling the rate of progress, in a way reminiscent of how spectral gaps control the behavior of random walks on graphs.
Understanding when $\Theta_\ell$ is bounded away from zero, and how it behaves for concrete model classes, is an important direction for future work.

A second theme is robustness.
Section~\ref{robust} shows that the same analysis naturally supports a \emph{blind-spot} viewpoint:~if the loss  function contains extremely narrow regions where it dips far below its typical values, then these regions may have limited influence on an algorithm that explores the landscape through random subspace sampling.
This perspective does not replace classical notions of robustness (e.g., uniform stability or worst-case adversarial guarantees), but it captures a failure mode that is common in complex non-convex landscapes:~``too-good-to-be-true'' minima supported on vanishingly small regions.

Finally, apart from optimization and robustness, the extensive use of randomness in our walk suggests potential privacy-related properties; studying this connection is another interesting direction.
Given the continued growth of large-scale, non-convex optimization problems across machine learning and beyond, we believe it is important to develop algorithms that have guarantees under minimal assumptions, alongside interpretations that reflect the geometry of modern objective function landscapes.
\begingroup
\small
\newcommand{\etalchar}[1]{$^{#1}$}

\endgroup


\begin{thebibliography}{CGD{\etalchar{+}}23}



\bibitem[BBC{\etalchar{+}}09]{Cohn1} Ballinger, B., Blekherman, G., Cohn, H., Giansiracusa, N., Kelly, E., \& Schürmann, A. (2009). Experimental Study of Energy-Minimizing Point Configurations on Spheres. Experimental Mathematics, 18(3), 257–283. https://doi.org/10.1080/10586458.2009.10129052

\bibitem[BBC{\etalchar{+}}21]{Cohn2} Ballinger, B., Blekherman, G., Cohn, H., Giansiracusa, N., Kelly, E., \& Schürmann, A..  Point configurations minimizing harmonic energy on spheres.  MIT Dataset (2021), https://hdl.handle.net/1721.1/130937

\bibitem[BdlHV08]{bekka2008kazhdan}
B.~Bekka, P.~de~la Harpe, and A.~Valette.
 {\em Kazhdan's Property (T)}.
 New Mathematical Monographs. Cambridge University Press, 2008.

\bibitem[BZA24]{Bendokat_2024}
Thomas Bendokat, Ralf Zimmermann, and P.-A. Absil.
 A Grassmann manifold handbook:~basic geometry and computational
  aspects.
 {\em Advances in Computational Mathematics}, 50(1), January 2024.

\bibitem[CGD{\etalchar{+}}23]{10.1145/3585385}
Antonio~Emanuele Cin\`{a}, Kathrin Grosse, Ambra Demontis, Sebastiano Vascon,
  Werner Zellinger, Bernhard~A. Moser, Alina Oprea, Battista Biggio, Marcello
  Pelillo, and Fabio Roli.
\newblock Wild patterns reloaded: A survey of machine learning security against
  training data poisoning.
\newblock {\em ACM Comput. Surv.}, 55(13s), Jul 2023.

\bibitem[CSV17]{10.1145/3055399.3055491}
Moses Charikar, Jacob Steinhardt, and Gregory Valiant.
\newblock Learning from untrusted data.
\newblock In {\em Proceedings of the 49th Annual ACM SIGACT Symposium on Theory
  of Computing}, STOC 2017, page 47–60, New York, NY, USA, 2017. Association
  for Computing Machinery.

\bibitem[DHS11]{JMLR:v12:duchi11a}
John Duchi, Elad Hazan, and Yoram Singer.
\newblock Adaptive subgradient methods for online learning and stochastic
  optimization.
\newblock {\em Journal of Machine Learning Research}, 12(61):2121--2159, 2011.

\bibitem[GQ20]{gallier2020differential}
J.~Gallier and J.~Quaintance.
\newblock {\em Differential Geometry and Lie Groups: A Computational
  Perspective}.
\newblock Geometry and Computing. Springer International Publishing, 2020.

\bibitem[G98]{Gillman}
D.~Gillman.
\newblock {\em A Chernoff Bound for Random Walks on Expander Graphs},
\newblock SIAM Journal on Computing, 27(4):1203--1220, 1998.


\bibitem[GSL21]{DBLP:conf/aistats/GowerSL21}
Robert~M. Gower, Othmane Sebbouh, and Nicolas Loizou.
\newblock {SGD} for structured nonconvex functions: Learning rates,
  minibatching and interpolation.
\newblock In Arindam Banerjee and Kenji Fukumizu, editors, {\em The 24th
  International Conference on Artificial Intelligence and Statistics, {AISTATS}
  2021, April 13-15, 2021, Virtual Event}, volume 130 of {\em Proceedings of
  Machine Learning Research}, pages 1315--1323. {PMLR}, 2021.

\bibitem[Hel22]{helgason2022groups}
S.~Helgason.
\newblock {\em Groups and Geometric Analysis: Integral Geometry, Invariant
  Differential Operators, and Spherical Functions}.
\newblock Mathematical Surveys and Monographs. American Mathematical Society,
  2022.

\bibitem[KB14]{Kingma2014AdamAM}
Diederik~P. Kingma and Jimmy Ba.
\newblock Adam: A method for stochastic optimization.
\newblock {\em CoRR}, abs/1412.6980, 2014.

\bibitem[KR23]{khaled2022better}
Ahmed Khaled and Peter Richt{\'a}rik.
\newblock Better theory for {SGD} in the nonconvex world.
\newblock {\em Transactions on Machine Learning Research}, 2023.
\newblock Survey Certification.

\bibitem[Mar73]{margulis1973explicit}
G.~Margulis.
\newblock Explicit constructions of concentrators.
\newblock {\em Problemy Peredachi Informatsii}, 9(4):71--80, 1973.

\bibitem[MB11]{NIPS2011_40008b9a}
Eric Moulines and Francis Bach.
\newblock Non-asymptotic analysis of stochastic approximation algorithms for
  machine learning.
\newblock In J.~Shawe-Taylor, R.~Zemel, P.~Bartlett, F.~Pereira, and K.Q.
  Weinberger, editors, {\em Advances in Neural Information Processing Systems},
  volume~24. Curran Associates, Inc., 2011.

\bibitem[MV06]{DBLP:conf/ants/MillerV06}
Stephen~D. Miller and Ramarathnam Venkatesan.
\newblock Spectral analysis of pollard rho collisions.
\newblock In Florian Hess, Sebastian Pauli, and Michael~E. Pohst, editors, {\em
  Algorithmic Number Theory, 7th International Symposium, ANTS-VII, Berlin,
  Germany, July 23-28, 2006, Proceedings}, volume 4076 of {\em Lecture Notes in
  Computer Science}, pages 573--581. Springer, 2006.

\bibitem[Nac76]{nachbin1976haar}
L.~Nachbin.
\newblock {\em The Haar Integral}.
\newblock University series in higher mathematics. R. E. Krieger Publishing
  Company, 1976.

\bibitem[PSBR20]{prasad2020robust}
Adarsh Prasad, Arun~Sai Suggala, Sivaraman Balakrishnan, and Pradeep Ravikumar.
\newblock Robust estimation via robust gradient estimation.
\newblock {\em Journal of the Royal Statistical Society: Series B (Statistical
  Methodology)}, 82(3):601--627, 2020.

\bibitem[RL94]{rogawski1994discrete}
J.D. Rogawski and A.~Lubotzky.
\newblock {\em Discrete Groups, Expanding Graphs and Invariant Measures}.
\newblock Progress in Mathematics. Birkh{\"a}user Basel, 1994.

\bibitem[Rud76]{rudin}
Walter Rudin.
\newblock {\em Principles of Mathematical Analysis, Third edition}.
\newblock {International Series in Pure and Applied Mathematics.  McGraw-Hill, 1976}.

\bibitem[Sep06]{sepanski2006compact}
M.R. Sepanski.
\newblock {\em Compact Lie Groups}.
\newblock Graduate Texts in Mathematics. Springer New York, 2006.

\bibitem[SHK{\etalchar{+}}14]{JMLR:v15:srivastava14a}
Nitish Srivastava, Geoffrey Hinton, Alex Krizhevsky, Ilya Sutskever, and Ruslan
  Salakhutdinov.
\newblock Dropout: A simple way to prevent neural networks from overfitting.
\newblock {\em Journal of Machine Learning Research}, 15(56):1929--1958, 2014.

\bibitem[Wal13]{Wales}
David Wales
\newblock {\em Energy Landscapes}
\newblock Cambridge Molecular Science.  Cambridge University Press, 2013.

\bibitem[TCLY22]{10.1145/3551636}
Zhiyi Tian, Lei Cui, Jie Liang, and Shui Yu.
\newblock A comprehensive survey on poisoning attacks and countermeasures in
  machine learning.
\newblock {\em ACM Comput. Surv.}, 55(8), Dec 2022.

\end{thebibliography}
\end{document}